\documentclass[11pt]{article}
\usepackage{amsmath,amssymb,textcomp,stmaryrd,xifthen,psfrag,graphicx,color}
\usepackage{amsfonts,amsthm}
\usepackage{color}

\SetSymbolFont{stmry}{bold}{U}{stmry}{m}{n} 
\usepackage[T1]{fontenc}
\date{}

\newtheorem{theorem}{Theorem}
\newtheorem{lemma}[theorem]{Lemma}

\newtheorem{proposition}[theorem]{Proposition}
\theoremstyle{definition}

\newcommand{\vdual}[2]{(#1\hspace*{.5mm},#2)}

\newcommand{\norm}[3][]{#1\|#2#1\|_{#3}}

\def\div{{\rm div}}

\newcommand{\R}{\ensuremath{\mathbb{R}}}

\newcommand{\Ii}{\ensuremath{\mathcal{I}}}
\newcommand{\Jj}{\ensuremath{\mathcal{J}}}

\newcommand{\tT}{\ensuremath{\mathcal{T}}}

\newcommand{\sS}{\ensuremath{\mathcal{S}}}
\newcommand{\OO}{\ensuremath{\mathcal{O}}}

\newcommand{\pphi}{{\boldsymbol\phi}}
\newcommand{\ppsi}{{\boldsymbol\psi}}

\newcommand{\ssigma}{{\boldsymbol\sigma}}

\title{Space-time least-squares finite elements for parabolic equations
\thanks{Supported by Conicyt Chile through projects FONDECYT 1170672 and 11170050.}}
\author{Thomas F\"uhrer$^1$ \and Michael Karkulik$^2$}
\date{%
    \small $^1$Facultad de Matem\'{a}ticas, Pontificia Universidad Cat\'{o}lica de Chile, Santiago, Chile\\%
    $^2$Departamento de Matem\'{a}tica,
    Universidad T\'{e}cnica Federico Santa Mar\'{i}a, Valpara\'{i}so, Chile\\[2ex]%
}
\begin{document}
\maketitle
\begin{abstract}

We present a space-time least squares finite element method for the heat equation.
It is based on residual minimization in $L^2$ norms in space-time of an equivalent first order system.
This implies that (i) the resulting bilinear form is symmetric and coercive and
hence any conforming discretization is uniformly stable, (ii)
stiffness matrices are symmetric, positive definite, and sparse,
(iii) we have a local a-posteriori error estimator for free.
In particular, our approach features full space-time adaptivity.
We also present a-priori error analysis on simplicial space-time meshes which are highly structured.
Numerical results conclude this work.

\bigskip
\noindent
{\em Key words}: Parabolic PDEs, space-time finite element methods, stability.\\
\noindent
{\em AMS Subject Classification}: 35K20, 65M12, 65M15, 65M60.
\end{abstract}
\section{Introduction}
By now, Galerkin finite element methods are ubiquitous for the numerical approximation of
elliptic partial differential equations. For the numerical solution of initial-boundary value problems
for parabolic partial differential equations
\begin{align*}
  \partial_t u + \mathcal{L} u &= f \text{ in } (0,T)\times \Omega,\\
  u &= 0 \text{ on } \partial\Omega,\\
  u(0) &= u_0 \text{ on } \Omega,
\end{align*}
it is then quite natural to
employ finite element methods only for the spatial part of the PDE,
and discretize the resulting system of ODEs by a time stepping method such as implicit Euler.
The derivation and error analysis of such semi-discretizations is by now standard textbook knowledge, cf.~\cite{Thomee}.
An advantage of time stepping schemes is that their oblivious nature allows for optimal storage requirements
if one is only interested in the final state. As soon as one considers problems where the entire history of the evolution problem
is of interest, such as control of PDE~\cite{CaraMS_17}, optimal control with PDE constraints~\cite{GunzburgerK_11},
or data assimilation~\cite{BurmanO_18},
this advantage becomes less beneficial. Furthermore, the flexibility of time stepping schemes with respect
to space-time local mesh refinement is very limited,
such that possible local space-time singularities prevent the optimal usage of computational resources.
Moreover, following~\cite{DouglasD_70}, quasi-optimality results such as a C\'ea's lemma are not available
for time-stepping schemes.
As was pointed out in~\cite{TantardiniV_16}, this has two mayor implications.
First, a-priori error bounds for time-stepping schemes are only of asymptotic nature and
do not cover the entire computational range as for Galerkin finite element methods. Second,
the established theory on convergence of adaptive finite element methods also relies heavily on quasi-optimality and
does therefore not carry over to time stepping schemes.
Finally, when it comes to the development of parallel solvers,
the sequentiality of time stepping schemes imposes severe difficulties.
For this and various other reasons, simultaneous space-time finite element discretizations, where
time is treated as just another spatial variable, have been proposed in recent years.
They all rely basically on the standard well-posed variational space-time formulation of parabolic
equations, cf.~\cite[Ch.~XVIII, $\mathsection$ 3]{DautrayL_92}, see also~\cite[Ch.~5]{SchwabS_09}.
This variational formulation is of Petrov--Galerkin type.
Uniform stability in the discretization parameter for pairs of discrete trial- and test-spaces
is therefore an issue. This issue turns out to be non-trivial and
might be identified as the main obstacle in obtaining a flexible space-time finite element method.
There are various works considering this problem.
Recently in~\cite{Andreev_13, Andreev_14,StevensonW_19}, using minimal residual
Petrov-Galerkin discretizations, uniform stability is obtained
for discrete spaces with non-uniform but, still, global time steps.
Another approach, taken in~\cite{Steinbach_15}, already allows for general simplicial space-time meshes,
but uniform stability was shown only with respect to a weaker, mesh-dependent norm. 
Unfortunately, uniform stability in the natural, mesh-independent energy norm
in this setting is out of reach, cf.~\cite[Remark~3.5]{StevensonW_19}.
We also mention that this approach can be extended to mesh- and degree-dependent norms in an $hp$-context~\cite{DevaudS_18}.
Then, in~\cite{SteinbachZ_18}, in the case of homogeneous initial conditions, the authors obtain an coercive Galerkin formulation
of the heat equation which involves the computation of a Hilbert type transform of test functions.

In the paper at hand, we reconsider the development of a space-time discretization for parabolic equations.
For clarity of presentation we focus exclusively on the heat equation $\mathcal{L} = -\Delta$,
although we have no reason to believe that our approach
does not carry over to general elliptic spatial differential operators of second order.
In order to effectively circumvent the problems we identified above, our method will be based on
the minimization of the space-time least-squares functional
\begin{align*}
  j(u,\ssigma) := \int_0^T \| \partial_t u - \div\, \ssigma - f \|_{L^2(\Omega)}^2 + \| \ssigma - \nabla u \|_{L^2(\Omega)}^2  \,dt
    + \|u(0) - u_0 \|_{L^2(\Omega)}^2
\end{align*}
over appropriately chosen spaces. We stress that this particular functional was already mentioned in~\cite[Ch.~9.1.4]{BochevG_09}.
Moreover, we note that a time-stepping scheme based on a
least-squares functional of the above type (but on local time slices) was developed in~\cite{StarkeM_01,StarkeM_02}.

Our motivation to consider space-time least-squares finite element schemes for solving parabolic problems are their attractive
properties:
\begin{itemize}
  \item \emph{Uniform stability:}
    The numerical method is uniformly stable for \emph{any} choice of conforming subspaces, i.e., the discrete
    $\inf$--$\sup$ constant is independent of the approximation space. 
    Particularly, this enables the use of arbitrary space-time meshes.
  \item \emph{Built-in adaptivity:}
    The least-squares functional evaluated in a discrete solution is equivalent 
    to the error between exact and discrete solution (in some norm).
    Since all norms in the functional are of $L^2$ type this allows to easily localize them into (space-time) element
    contributions which can be used to steer a standard adaptive algorithm.
  \item \emph{Symmetric, positive definite, and sparse algebraic systems:}
    The bilinear form associated with the least-squares functional is symmetric and coercive, and therefore the stiffness
    matrix of the discretized problem is symmetric and positive definite. This enables the use of standard solvers, e.g., the
    preconditioned CG method. Moreover, we avoid the use of negative order Sobolev norms, so that by using locally
    supported basis functions the resulting stiffness matrix is sparse.
\end{itemize}
In Section~\ref{sec:main} we introduce and analyze the space-time least-squares functional for a first-order
reformulation of the heat equation.
We show that the right space associated to the problem consists of pairs of functions $(u,\ssigma)$
with component $u$ in the standard energy space for parabolic problems and $\ssigma$ in $L^2$, with the additional
restriction that $\partial_t u - \div\,\ssigma$ in $L^2$.
The natural norm in this space is stronger than the standard energy norm for the heat equation.
(This is similar to least-squares methods for elliptic problems, where one assumes that $\div\,\ssigma$ is in $L^2$
instead of a negative order Sobolev space.)
Since one of our aims is also to provide an easy-to-implement numerical method, we consider in Section~\ref{sec:approximation}
one of the simplest approximation spaces, that is, piecewise affine and globally continuous functions (``low order finite element spaces'')
for both variables on a space-time mesh.
We present a-priori error analysis for simplicial space-time meshes which are uniform and highly structured.
Convergence rates are shown provided that the solution is sufficiently regular.
The final Section~\ref{sec:examples} deals with an extensive study of numerical examples for problems with
spatial domains in one respectively two dimensions.

To close the introduction we like to mention the recent works~\cite{LangerMN_16,Moore_18,KimPS_18,MNST19}.
There is also plenty of literature on time stepping methods using least-squares FEMs. For an overview we refer
to~\cite[Ch.~9]{BochevG_09}.
In our recent work~\cite{FuehrerKarkulik19} we improved the existing literature on time stepping least-squares method and showed optimal
a-priori error bounds without relying on the so-called splitting property.

\section{Sobolev and Bochner spaces}
For a bounded (spatial) Lipschitz domain $\Omega\subset\R^d$ we consider the standard Lebesgue and Sobolev spaces
$L^2(\Omega)$ and $H^k(\Omega)$ for $k\geq 1$ with the standard norms. The space $H^1_0(\Omega)$ consists of all $H^1(\Omega)$ functions
with vanishing trace on the boundary $\partial\Omega$.
We define $H^{-1}(\Omega):= H^1_0(\Omega)'$ and $H^{-1}_0(\Omega) := H^1(\Omega)'$ as topological duals
with respect to the extended $L^2(\Omega)$ scalar product $\vdual{\cdot}{\cdot}_\Omega$.

For a fixed, bounded time interval $J=(0,T)$ and a Banach space $X$ we will use the space $L^2(X)$ of functions
$f:J\rightarrow X$ which are strongly measurable with respect to the Lebesgue measure $ds$ on $\R$ and
\begin{align*}
  \| f \|_{L^2(X)}^2 := \int_J \| f(s) \|_X^2\,ds < \infty.
\end{align*}
A function $f\in L^2(X)$ is said to have a weak time-derivative $f'\in L^2(X)$, if
\begin{align*}
  \int_J f'(s)\cdot\varphi(s)\,ds = -\int_J f(s)\cdot \varphi'(s)\,ds\quad\text{ for all test functions }\varphi \in C^\infty_0(J).
\end{align*}
We then define the Sobolev-Bochner space $H^k(X)$ of functions in $L^2(X)$ whose weak derivatives $f^{(\alpha)}$ of all orders $|\alpha|\leq k$
exist, endowed with the norm
\begin{align*}
  \| f \|_{H^k(X)}^2 := \sum_{|\alpha|\leq k} \| f^{(\alpha)} \|_{L^2(X)}^2.
\end{align*}
If we denote by $C(\overline J;X)$ the space of continuous functions $f:\overline{J}\rightarrow X$
endowed with the natural norm, then we have the following well-known result, cf.~\cite[Thm.~25.5]{Wloka_87}.
\begin{lemma}\label{lem:ttrace}
  Let $X\hookrightarrow H \hookrightarrow X'$ be a Gelfand triple.
  Then, the embedding
  \begin{align*}
    L^2(X)\cap H^1(X') \hookrightarrow C(\overline J;H)
  \end{align*}
  is continuous.
\end{lemma}
We will have to interchange spatial and temporal derivatives in Bochner spaces.
To that end, we will employ the following lemma.
\begin{lemma}\label{lem:schwarz}
  Let $L:X\rightarrow Y$ be a linear and bounded operator between two Banach spaces $X$ and $Y$
  and $u\in H^1(X)$. Then it holds $Lu \in H^1(Y)$ and $(Lu)' = L(u')$.
\end{lemma}
\begin{proof}
  It is well known, cf.~\cite[V.5, Cor~2]{Yosida} that if $f:J\rightarrow X$ is Bochner integrable, then $Lf: J\rightarrow Y$ is Bochner integrable
  and
  \begin{align*}
    L \int_J f(t)\,dt = \int_J (L f)(t)\,dt.
  \end{align*}
  For $u\in H^1(X)$, we calculate for any $\varphi\in C_0^\infty(J)$
  \begin{align*}
    \int_J(Lu)(t) \varphi'(t)\,dt = L \int_J u(t) \varphi'(t)\,dt = - L \int_J u'(t)\varphi(t)\,dt = -\int_J L(u')(t)\varphi(t)\,dt.
  \end{align*}
\end{proof}
\section{Least-squares formulation of the heat equation}\label{sec:main}
Let $\Omega\subset\R^d$ be a bounded (spatial) Lipschitz domain and $J=(0,T)$ a given finite time interval.
For two functions $f\in L^2(L^2(\Omega))$, and $u_0\in L^2(\Omega)$
we consider the problem to find $u\in L^2(H^1_0(\Omega))\cap H^1(H^{-1}(\Omega))$
such that
\begin{align}\label{eq:model}
  \begin{split}
  \partial_t u -\Delta u &= f\text{ in } J\times \Omega,\\
  u &= 0 \text{ on } J\times \partial\Omega,\\
  u(0) &= u_0 \text{ on } \Omega.
  \end{split}
\end{align}
It is important to mention that problem~\eqref{eq:model} is well posed even if $f\in L^2(H^{-1}(\Omega))$,
cf.~\cite[Thm.~23.A]{Zeidler_90}:
\begin{proposition}\label{prop:regularitymodel}
  Let $f\in L^2(H^{-1}(\Omega))$, $u_0\in L^2(\Omega)$. Then, the solution of~\eqref{eq:model} enjoys the stability
  estimate
  \begin{align*}
    \norm{u}{L^2(H_0^1(\Omega))\cap H^1(H^{-1}(\Omega))} \lesssim \norm{f}{L^2(H^{-1}(\Omega))} +
      \norm{u_0}{L^2(\Omega)}.
  \end{align*}
\end{proposition}

However if $f\in L^2(L^2(\Omega))$, then there holds the additional regularity $u\in L^2(H_0^1(\Omega))\cap H^1(L^2(\Omega))$.
For $f\in L^2(L^2(\Omega))$, $u_0\in L^2(\Omega)$ we define the least-squares functional
\begin{align}\label{eq:j}
  j(v,\ppsi) := \norm{\ppsi-\nabla v}{L^2(J\times\Omega)}^2 + \norm{\partial_t v-\div\,\ppsi-f}{L^2(J\times \Omega)}^2
  + \norm{v(0)-u_0}{L^2(\Omega)}^2.
\end{align}
The solution $u$ of~\eqref{eq:model} satisfies $j(u,\nabla u)=0$. 
The minimization of the functional $j$ then gives rise to the bilinear form
\begin{align*}
  b(u,\ssigma;v,\ppsi) := \vdual{\nabla u-\ssigma}{\nabla v-\ppsi}_{J\times\Omega} +
  \vdual{\partial_tu-\div\,\ssigma}{\partial_t v-\div\,\ppsi}_{J\times\Omega}
  +\vdual{u(0)}{v(0)}_{\Omega}.
\end{align*}
and the linear functional
\begin{align*}
  \ell(v,\ppsi) := \vdual{f}{\partial_tv-\div\,\ppsi}_{J\times\Omega} + \vdual{u_0}{v(0)}_{\Omega}.
\end{align*}
Define the Hilbert space
\begin{align*}
  U :=
  \left\{ (v,\ppsi)\mid v \in L^2(H^1_0(\Omega))\cap H^1(H^{-1}(\Omega)), \ppsi\in L^2(J\times\Omega), \partial_t v -\div\,\ppsi\in L^2(J\times\Omega) \right\}
\end{align*}
with its natural norm
\begin{align*}
  \| (v,\ppsi) \|_U^2 := \| v \|_{L^2(H^1_0(\Omega))}^2 + \| v \|_{H^1(H^{-1}(\Omega))}^2 + \| \ppsi \|_{L^2(J\times\Omega)}^2
  + \| \partial_t v -\div\,\ppsi \|_{L^2(J\times\Omega)}^2.
\end{align*}
Consider the problem to
\begin{align}\label{eq:ls}
  \text{find } (u,\ssigma)\in U \text{ such that } b(u,\ssigma;v,\ppsi) = \ell(v,\ppsi)\quad\text{ for all } (v,\ppsi)\in U.
\end{align}
The solution $u$ of~\eqref{eq:model} and $\ssigma:=\nabla u$ solve~\eqref{eq:ls}.
We can show the following.
\begin{lemma}\label{lem:ell}
  The bilinear form $b$ is bounded and coercive on $U$, and the linear functional $\ell$ is bounded on $U$.
\end{lemma}
\begin{proof}
  Boundedness of $b$ and $\ell$ follows immediately by Cauchy-Schwarz and Lemma~\ref{lem:ttrace}.
  To show coercivity of $b$, let $(v,\ppsi)\in U$ be arbitrary.
  Then,
  \begin{align*}
    \partial_t v-\Delta v = \partial_t v -\div\,\ppsi + \div\,(\ppsi - \nabla v),
  \end{align*}
  where the right-hand side is taken in $L^2(H^{-1}(\Omega))$. Due to the well-posedness of the parabolic problem
  (Proposition~\ref{prop:regularitymodel} with $f=\partial_tv -\div\ppsi+\div(\ppsi-\nabla v)$, $u_0 = v(0)$) and obvious bounds,
  \begin{align*}
    \norm{v}{L^2(H^1_0(\Omega))\cap H^1(H^{-1}(\Omega))} &\lesssim \norm{\partial_t v -\div\,\ppsi}{L^2(H^{-1}(\Omega))}\\
    &\qquad+ \norm{\div\,(\ppsi - \nabla v)}{L^2(H^{-1}(\Omega))} + \norm{v(0)}{L^2(\Omega)}\\
    &\lesssim \norm{\partial_t v -\div\,\ppsi}{L^2(J\times\Omega)}\\
    &\qquad+ \norm{\ppsi - \nabla v}{L^2(J\times \Omega)} + \norm{v(0)}{L^2(\Omega)}.
  \end{align*}
  The triangle inequality and the last bound also yields that
  \begin{align*}
    \norm{\ppsi}{L^2(J\times\Omega)} &\leq \norm{\ppsi-\nabla v}{L^2(J\times\Omega)} + \norm{\nabla v}{L^2(J\times\Omega)}\\
    &\leq \norm{\ppsi-\nabla v}{L^2(J\times\Omega)} + \norm{v}{L^2(H^1_0(\Omega))}\\
    &\lesssim \norm{\partial_t v -\div\,\ppsi}{L^2(J\times\Omega)}
    + \norm{\ppsi - \nabla v}{L^2(J\times \Omega)} + \norm{v(0)}{L^2(\Omega)}.
  \end{align*}
  Hence
  \begin{align*}
    \norm{(u,\ppsi)}{U}^2
    &= \norm{v}{L^2(H^1_0(\Omega))\cap H^1(H^{-1}(\Omega))}^2 +
    \norm{\ppsi}{L^2(J\times\Omega)}^2 + \norm{\partial_t v - \div\,\ppsi}{L^2(J\times \Omega)}^2\\
    &\lesssim \norm{\partial_t v -\div\,\ppsi}{L^2(J\times\Omega)}^2 +
    \norm{\ppsi - \nabla v}{L^2(J\times \Omega)}^2 + \norm{v(0)}{L^2(\Omega)}^2\\
    &= b(v,\ppsi;v,\ppsi),
  \end{align*}
  which proves coercivity.
\end{proof}
The last lemma immediately implies the first main result of this work.
\begin{theorem}\label{thm:ls:wp}
  Problem~\eqref{eq:ls} is well-posed. Furthermore, if $U_h\subset U$ is a closed subspace, then the problem to
  \begin{align}\label{eq:ls:disc}
    \text{find } (u_h,\ssigma_h)\in U_h \text{ such that } b(u_h,\ssigma_h;v_h,\ppsi_h) = \ell(v_h,\ppsi_h)\quad\text{ for all } (v_h,\ppsi_h)\in U_h.
  \end{align}
  is well-posed and there holds the quasi-optimality
  \begin{align*}
    \norm{(u-u_h,\ssigma-\ssigma_h)}{U} \lesssim \min_{(w_h,\pphi_h)\in U_h} \norm{(u-w_h,\ssigma-\pphi_h)}{U}
  \end{align*}
\end{theorem}
\begin{proof}
  Well-posedness follows immediately from Lemma~\ref{lem:ell} and the Lax--Milgram theorem.
  The quasi-optimality result is a standard consequence.
\end{proof}
\section{Numerical approximation by finite elements}\label{sec:approximation}
Let $\tT_h$ be a simplicial and admissible partition of the space-time cylinder $J\times\Omega$.
By admissible, we mean that there are no hanging nodes. Define
\begin{align*}
  \sS^1(\tT_h) &= \left\{ u \in C(J\times\Omega) \mid u|_K \text{ a polynomial of degree at most } 1 \text{ for all
  } K\in \tT_h \right\},\\
  \sS^1_0(\tT_h) &= \left\{ u \in \sS^1(\tT_h) \mid u = 0 \text{ on } \overline J\times\partial\Omega  \right\}.
\end{align*}
Note that $\sS^1_0(\tT_h)$ is a subspace of $L^2(J;H^1_0(\Omega))\cap H^1(J;H^{-1}(\Omega))$,
and $[\sS^1(\tT_h)]^d$ is a subspace of $L^2(J\times\Omega)^d$.
Furthermore, if $v_h\in\sS^1_0(\tT_h)$ and $\ppsi_h\in[\sS^1(\tT_h)]^d$, then
$\partial_tv_h - \div\,\ppsi_h\in L^2(J\times \Omega)$.
Hence, we can define the discrete conforming subspace
\begin{align*}
  U_h = \sS^1_0(\tT_h)\times [\sS^1(\tT_h)]^d \subset U.
\end{align*}
A finite-element approximation of Problem~\eqref{eq:ls} is then given by~\eqref{eq:ls:disc}.
\subsection{A-priori convergence theory}
Due the quasi-optimality of Theorem~\ref{thm:ls:wp}, in order to provide a-priori error analysis
it suffices to analyze the approximation properties of $U_h$.
In the present section, we will show such approximation results, provided that the space-time mesh $\tT_h$ is
uniform and structured. By structured, we mean that $\tT_h$ is obtained from a tensor-product mesh $J_h\otimes\Omega_h$ via refinement
of the tensor-product space-time cylindrical elements into simplices, cf. Section~\ref{sec:spint}.
First, in Section~\ref{sec:tensorint}, we will obtain auxiliary results for space-time interpolation on the tensor product mesh $J_h\otimes\Omega_h$.
To that end, let $J_h=\{ j_1,j_2,\dots \}$ with $j_k=(t_k,t_{k+1})$ be a partition of the time interval $J$ into subintervals of length $h$,
and let $\Omega_h = \{\omega_1,\omega_2,\dots\}$ be a partition of physical space $\Omega\subset\R^d$ into $d$-simplices of diameter $h$.
Then, define the (discrete) spaces
\begin{align*}
  \sS^1(J_h;X) &:= \left\{ u\in C(J;X) \mid u|_j \text{ is affine in time with values in } X \text{ for all } j\in
  J_h \right\},\\
  \sS^1(\Omega_h) &:= \left\{ u\in C(\Omega) \mid u|_\omega \text{ a polynomial of degree at most } 1 \text{ for all }\omega\in\Omega_h \right\},\\
  \sS^1_0(\Omega_h) &:= \left\{ u \in\sS^1(\Omega_h) \mid u = 0 \text{ on } \partial\Omega \right\}.
\end{align*}
\subsubsection{Space-time interpolation on tensor product meshes}\label{sec:tensorint}
First we will show that a function $u$ can be approximated in the norm of
$L^2(H^1_0(\Omega))\cap H^1(H^{-1}(\Omega))$ to first order by
discrete functions in $\sS^1(J_h;\sS^1_0(\Omega_h))$, given that $u$ possesses some additional regularity
which is in accordance with regularity results for the heat equation.
To that end, we consider first a discretization only in time, given by the piecewise linear interpolation operator
\begin{align*}
  \Ii_h^\otimes u(t_k) := u(t_k).
\end{align*}
\begin{lemma}\label{thm:tp:lagrange}
  Let
  $\Ii_h^\otimes:C(\overline J;L^2(\Omega))\rightarrow \sS^1(J_h;L^2(\Omega))$
  be the piecewise linear interpolation operator. If $u\in H^1(X)$ for some Hilbert space $X$, it holds
  \begin{align*}
    \| u-\Ii_h^\otimes u \|_{L^2(X)} &\lesssim h \| u \|_{H^1(X)}.
  \end{align*}
  If $u\in H^2(X)$ for some Hilbert space $X$, it holds
  \begin{align*}
    \| (u-\Ii_h^\otimes u)' \|_{L^2(X)} &\lesssim h \| u \|_{H^2(X)}.
  \end{align*}
\end{lemma}
\begin{proof}
  For $t\in(t_j,t_{j+1})$, we have according to~\cite[Prop.~2.5.9]{HvNVW_16} for $u\in H^1(X)$
  \begin{align*}
    u(t) = u(t_j) + \int_{t_j}^t u'(s)\,ds = u(t_{j+1}) - \int_t^{t_{j+1}} u'(s)\,ds
  \end{align*}
  in $X$. Hence,
  \begin{align*}
    \Ii_h^\otimes u (t) = \frac{u(t_{j+1}) (t-t_j) + u(t_j)(t_{j+1}-t)}{h}
    = u(t) + \frac{(t-t_j)\int_t^{t_{j+1}} u'(s)\,ds -(t_{j+1}-t)\int_{t_j}^t u'(s)\,ds}{h},
  \end{align*}
  and we conclude with H\"older that
  \begin{align*}
    \| u(t) - \Ii_h^\otimes u (t) \|_{X}^2 \lesssim h\int_{t_j}^{t_{j+1}} \| u'(s) \|_{X}^2\,ds.
  \end{align*}
  Another integration in $t$ shows the first of the stipulated estimates.
  Likewise, for $t\in(t_j,t_{j+1})$ we can apply two times~\cite[Prop.~2.5.9]{HvNVW_16} for $u\in H^2(X)$ to see
  \begin{align*}
    u(t_{j+1}) &= u(t)+(t_{j+1}-t)u'(t) + \int_t^{t_{j+1}}\int_t^s u''(r)\,dr\,ds,\\
    u(t_j) &= u(t)+(t_j-t)u'(t) + \int_t^{t_j}\int_t^s u''(r)\,dr\,ds
  \end{align*}
  in $X$. Hence,
  \begin{align*}
    (u-\Ii_h^\otimes u)'(t) = u'(t) - \frac{u(t_{j+1}) - u(t_j)}{h}
    = \frac1h\int_t^{t_j}\int_t^s u''(r)\,dr\,ds - \frac1h\int_t^{t_{j+1}}\int_t^s u''(r)\,dr\,ds,
  \end{align*}
  and we conclude with H\"older
  \begin{align*}
    \| (u-\Ii_h^\otimes u)'(t) \|_X^2 \lesssim h\int_{t_j}^{t_{j+1}} \| u''(r) \|_{X}^2\,dr
  \end{align*}
  Another integration in $t$ shows the second of the stipulated estimates.
\end{proof}
Next, we consider fully discrete interpolation operators. In order to analyze their approximation properties,
we will compare them to the semi-discrete operator $\Ii_h^\otimes$. We will employ the $L^2(\Omega)$-orthogonal projections
\begin{align*}
  \Pi_h:L^2(\Omega)&\rightarrow \sS^1(\Omega_h),\\
  \Pi_{0,h}:L^2(\Omega)&\rightarrow \sS^1_0(\Omega_h).
\end{align*}
There holds the approximation property
\begin{align}\label{eq:l2:approx}
  \| \Pi_hu-u \|_{L^2(\Omega)} \lesssim h \| u \|_{H^1(\Omega)}.
\end{align}
Furthermore, it is well-known that $\Pi_h$ is $H^1(\Omega)$-stable uniformly in $h$ on uniform meshes, i.e.,
\begin{align}\label{eq:l2:stab}
  \| \Pi_hu \|_{H^1(\Omega)} \lesssim \| u \|_{H^1(\Omega)}.
\end{align}
This, and the fact that $\Pi_h$ is a projection, implies the approximation estimate
\begin{align}\label{eq:l2:approx:h1}
  \| \Pi_hu-u \|_{H^1(\Omega)} \lesssim h \| u \|_{H^2(\Omega)}.
\end{align}
The statements~\eqref{eq:l2:approx}--\eqref{eq:l2:approx:h1} also hold if we replace $\Pi_h$ by $\Pi_{0,h}$.
Furthermore, it holds that
\begin{align}
  \| \Pi_{0,h}u-u \|_{L^2(\Omega)} &\lesssim h^2 \| u \|_{H^2(\Omega)},
  \label{eq:l20:approx}\\
  \| \Pi_{0,h}u-u \|_{H^{-1}(\Omega)} &\lesssim h^2 \| u \|_{H^1_0(\Omega)},\label{eq:l20:approx:1}
\end{align}
where the second estimate follows from a duality argument.
\begin{theorem}\label{thm:tp}
  Define the operator
  \begin{align*}
    \Jj_h^\otimes: C(\overline J;L^2(\Omega))\rightarrow \sS^1(J_h;\sS^1(\Omega_h))
  \end{align*}
  by $\Jj_h^\otimes := \Ii_h^\otimes\circ (Id\otimes\Pi_h)$. Then it holds
  \begin{align*}
    \| u-\Jj_h^\otimes u \|_{L^2(L^2(\Omega))} &\lesssim h ( \| u \|_{H^1(L^2(\Omega))} + \| u \|_{L^\infty(H^1(\Omega))} ),\\
    \| u-\Jj_h^\otimes u \|_{L^2(H^1(\Omega))} &\lesssim h ( \| u \|_{H^1(H^1(\Omega))} + \| u \|_{L^\infty(H^2(\Omega))} ).
  \end{align*}
  Define the operator
  \begin{align*}
    \Jj_{0,h}^\otimes: C(\overline J;L^2(\Omega))\rightarrow \sS^1(J_h;\sS^1_0(\Omega_h)),
  \end{align*}
  by $\Jj_{0,h}^\otimes := \Ii_h^\otimes\circ (Id\otimes\Pi_{0,h})$. Then it holds
  \begin{align*}
    \| u-\Jj_{0,h}^\otimes u \|_{L^2(H^1_0(\Omega))} &\lesssim h ( \| u \|_{H^1(H^1_0(\Omega))} + \| u \|_{L^\infty(H^2(\Omega))} ),\\
    \| (u-\Jj_{0,h}^\otimes u)' \|_{L^2(H^{-1}(\Omega))} &\lesssim h ( \| u \|_{H^2(H^{-1}(\Omega))} + \| u \|_{L^\infty(H^1_0(\Omega))} ),\\
    \| (u-\Jj_{0,h}^\otimes u)' \|_{L^2(L^2(\Omega))} &\lesssim h ( \| u \|_{H^2(L^2(\Omega))} + \| u \|_{L^\infty(H^2(\Omega))} ).
  \end{align*}
\end{theorem}
\begin{proof}
  We will first prove the statement for the operator $\Jj_h$. Note that for $t\in(t_j,t_{j+1})$ it holds
  \begin{align}\label{eq:thm:tp:1}
    \Jj_h^\otimes u(t) = \Ii_h^\otimes u(t) + \frac{ [\Pi_h(u(t_{j+1}))-u(t_{j+1})] (t-t_j) + [\Pi_h(u(t_{j}))-u(t_{j})] (t_{j+1}-t)}{h}.
  \end{align}
  Using the approximation property~\eqref{eq:l2:approx} we conclude
  \begin{align*}
    \| u(t)-\Jj_h^\otimes u(t) \|_{L^2(\Omega)}
    &\leq \| u(t) - \Ii_h^\otimes u (t) \|_{L^2(\Omega)}
    + \sum_{k=0,1} \| \Pi_h(u(t_{j+k}))-u(t_{j+k}) \|_{L^2(\Omega)}\\
    &\leq \| u(t) - \Ii_h^\otimes u (t) \|_{L^2(\Omega)}
    + h\sum_{k=0,1} \| u(t_{j+k}) \|_{H^1(\Omega)}.
  \end{align*}
  Another integration in $t$ and application of Theorem~\ref{thm:tp:lagrange} with $X=L^2(\Omega)$ shows the first of the stipulated estimates.
  Likewise, from~\eqref{eq:thm:tp:1} we conclude with the approximation property~\eqref{eq:l2:approx:h1}
  \begin{align*}
    \| u(t)-\Jj_h^\otimes u(t) \|_{H^1(\Omega)} &\leq \| u(t) - \Ii_h^\otimes u (t) \|_{H^1(\Omega)}
    + h\sum_{k=0,1} \| u(t_{j+k}) \|_{H^2(\Omega)}.
  \end{align*}
  Another integration in $t$ and application of Theorem~\ref{thm:tp:lagrange} with $X=H^1(\Omega)$ shows the second of the stipulated estimates.
  To prove the statements for the operator $\Jj_{0,h}^\otimes$ we note that the first estimate follows as in the case of $\Jj_h^\otimes$,
  only replacing $\Pi_h$ by $\Pi_{0,h}$ and $H^1(\Omega)$ by $H^1_0(\Omega)$. Next,
  \begin{align}\label{thm:tp:eq1}
    (\Jj_{0,h}^\otimes u)'(t) = (\Ii_h^\otimes u)'(t) + \frac{ [\Pi_{0,h}(u(t_{j+1}))-u(t_{j+1})] - [\Pi_{0,h}(u(t_{j}))-u(t_{j})] }{h}.
  \end{align}
  We conclude with~\eqref{eq:l20:approx:1} that
  \begin{align*}
    \| u'(t)-(\Jj_{0,h}^\otimes u)'(t) \|_{H^{-1}(\Omega)}
    &\leq \| u'(t) - (\Ii_h^\otimes u)' (t) \|_{H^{-1}(\Omega)}
    + \frac1h\sum_{k=0,1} \| \Pi_{0,h}(u(t_{j+k}))-u(t_{j+k}) \|_{H^{-1}(\Omega)}\\
    &\leq \| u'(t) - (\Ii_h^\otimes u)' (t) \|_{H^{-1}(\Omega)}
    + h\sum_{k=0,1} \| u(t_{j+k}) \|_{H^1_0(\Omega)}.
  \end{align*}
  Another integration in $t$ and application of Theorem~\ref{thm:tp:lagrange} with $X = H^{-1}(\Omega)$ shows the second
  of the stipulated estimates for $\Jj_{0,h}^\otimes$.
  To show the third estimate we apply the same arguments, only this time using the approximation estimate~\eqref{eq:l20:approx}.
\end{proof}
\subsubsection{Space-time interpolation on simplicial meshes}\label{sec:spint}
The tensor-product mesh $J_h\otimes\Omega_h$ consists of elements which are space-time cylinders with
$d$-simplices from $\Omega_h$ as base. It is possible to construct from $J_h\otimes\Omega_h$ a simplicial, admissible mesh $\tT_h$,
following the recent work~\cite{NK_15}. To that end, suppose that the vertices of $\Omega_h$ are numbered like
$p_1,p_2, \dots, p_N$. An element $\omega\in\Omega_h$ can then be represented uniquely as the convex hull
\begin{align*}
  \omega = \textrm{conv}(p_{i^\omega_1}, \dots, p_{i^\omega_{d+1}}), \text{ with } i^\omega_k < i^\omega_\ell \text{ for } k<\ell.
\end{align*}
This ``local numbering'' of vertices is called \textit{consistent numbering} in the literature, cf.~\cite{Bey_00}.
An element $K=\omega\times j_k \in J_h\otimes\Omega_h$ can hence be written as convex hull
\begin{align*}
  K = \textrm{conv}(p_1', \dots, p_{d+1}',p_1'',\dots,p_{d+1}''),
\end{align*}
where
\begin{align*}
  p_\ell' = (p_{i^\omega_\ell},t_k), \quad p_\ell'' = (p_{i^\omega_\ell},t_{k+1}).
\end{align*}
We can split $K$ into $(d+1)$ different $d+1$-simplices
\begin{align}
  \begin{split}\label{int:simp:split}
  \textrm{conv}&(p_1',p_2',\dots, p_{d+1}',p_1'')\\
  \textrm{conv}&(p_2',\dots, p_{d+1}',p_1'',p_2'')\\
  &\vdots\\
  \textrm{conv}&(p_{d+1}',p_1'',p_2'',\dots,p_{d+1}''),
  \end{split}
\end{align}
and this way we obtain a simplicial triangulation $\tT_h$ of $J\times \Omega$.
There holds the following result from~\cite[Thm.~1]{NK_15}.
\begin{theorem}
  The simplicial partition $\tT_h$ is admissible.
\end{theorem}

In order to approximate a function in space-time by an element of $\sS^1(\tT_h)$, we note that
$J_h\otimes\Omega_h$ and $\tT_h$ have the same set of vertices, and hence
$\sS^1(J_h;\sS^1(\Omega_h))$ and $\sS^1(\tT_h)$, as well as $\sS^1(J_h;\sS^1_0(\Omega_h))$ and $\sS^1_0(\tT_h)$,
have the same degrees of freedom. We can therefore define operators
\begin{align*}
  \Jj_h:L^2(H^1(\Omega))\cap H^1(H^{-1}_0(\Omega))\rightarrow \sS^1(\tT_h),\\
  \Jj_{0,h} :L^2(H^1_0(\Omega))\cap H^1(H^{-1}(\Omega))\rightarrow \sS^1_0(\tT_h)
\end{align*}
by requiring $\Jj_h u$ to have the same values as $\Jj_h^\otimes u$ at all vertices, likewise for $\Jj_{0,h}$.
We will analyze these new operators by comparing them to their tensor product versions.
To that end, the following Lemma will be useful.
\begin{lemma}\label{thm:int:simp:aux}
  There holds
  \begin{align*}
    \| \nabla(\Jj_{h}^\otimes u)' \|_{L^2(L^2(\Omega))} \lesssim \| u \|_{H^1(H^1(\Omega))} + \| u \|_{L^\infty(H^2(\Omega))},\\
    \| \nabla(\Jj_{0,h}^\otimes u)' \|_{L^2(L^2(\Omega))} \lesssim \| u \|_{H^1(H^1_0(\Omega))} + \| u \|_{L^\infty(H^2(\Omega))}.
  \end{align*}
\end{lemma}
\begin{proof}
  We will show the second estimate, the first one follows analogously.
  Note that for $t\in(t_j,t_{j+1})$ we have due to~\eqref{thm:tp:eq1}
  \begin{align*}
    \| \nabla(\Jj_{0,h}^\otimes u)'(t) \|_{L^2(\Omega)} &\lesssim \frac1h\| u(t_{_j+1})-u(t_j) \|_{H^1_0(\Omega)}
    + \frac1h\sum_{k=0,1} \| \Pi_{0,h}(u(t_{j+k}))-u(t_{j+k}) \|_{H^1_0(\Omega)}\\
    &\lesssim h^{-1/2}\| u' \|_{L^2(t_j,t_{j+1};H^1_0(\Omega))}
    + \sum_{k=0,1} \| u(t_{j+k}) \|_{H^2(\Omega)}.
  \end{align*}
  An integration in $t$ shows the result.
\end{proof}
\begin{lemma}\label{lem:int:simp}
  Let $K\in \tT_h$. Then,
  \begin{align*}
    \|\nabla(\Jj_{0,h}^\otimes-\Jj_{0,h})u\|_{L^2(K)} + \|(\Jj_{0,h}^\otimes u-\Jj_{0,h}u)'\|_{L^2(K)}
    &\lesssim h^{-1} \|(\Jj_{0,h}^\otimes-\Jj_{0,h})u\|_{L^2(K)} 
    \\ 
    &\lesssim h \|\nabla(\Jj_{0,h}^\otimes u)'\|_{L^2(K)}
  \end{align*}
  and
  \begin{align*}
    \|\nabla(\Jj_{h}^\otimes-\Jj_{h})u\|_{L^2(K)} + \|(\Jj_{h}^\otimes u-\Jj_{h}u)'\|_{L^2(K)}
    &\lesssim h^{-1} \|(\Jj_{h}^\otimes-\Jj_{h})u\|_{L^2(K)} 
    \lesssim h \|\nabla(\Jj_{h}^\otimes u)'\|_{L^2(K)}.
  \end{align*}
\end{lemma}
\begin{proof}
  We will only show the estimates involving $\Jj_{0,h}$, as the ones involving $\Jj_h$ follow the same lines.
  The first estimate follows from a standard inverse inequality on polynomial spaces.
  To see the second, we write $K = \mathrm{conv}\{z_1,\dots,z_{d+2}\}$ and
  \begin{align*}
  \Jj_{0,h}^\otimes u|_K = \sum_{j=1}^{d+2} \alpha_j \eta_j + \sum_{j=1}^d \beta_j \nu_j \in
    \mathrm{span}\{1,t,x_j,x_jt\,:\,j=1,\dots,d\},
  \end{align*}
  where $\eta_1,\dots,\eta_{d+2}$ are the hat functions associated to the vertices $z_j$ of $K$.
  Here, $1,t,x_j,x_jt$ stand for the functions $(t,x)\mapsto 1$, $(t,x)\mapsto t$, $(t,x)\mapsto x_j$, $(t,x)\mapsto x_jt$.
  Moreover, we choose $\nu_j$ such that these functions vanish in the vertices of $K$,  $\|\nu_j\|_{L^\infty(K)}
  \lesssim 1$, and $\nabla \nu_j' \simeq h^{-2}$.
  By definition $J_{0,h}u$ is affine on $K$ and takes the same values as $J_{0,h}^\otimes u$ in the vertices of $K$. Therefore
  \begin{align*}
    (\Jj_{0,h}^\otimes-\Jj_{0,h})u|_K = \sum_{j=1}^d \beta_j \nu_j.
  \end{align*}
  Then, scaling arguments, norm equivalence 
  and $\nabla(\nu_j)'\simeq h^{-2}$ show that
  \begin{align*}
    \| (\Jj_{0,h}^\otimes-\Jj_{0,h})u \|_{L^2(K)} &\lesssim |K|^{1/2} \sum_{j=1}^d |\beta_j|
    = h^2 \frac{|K|^{1/2}}{h^2} \sum_{j=1}^d |\beta_j| \simeq h^2 \|\nabla(J_{0,h}^\otimes u)'\|_{L^2(K)},
  \end{align*}
  which finishes the proof.
\end{proof}

\begin{theorem}\label{thm:int:simp}
  There holds
  \begin{align*}
    \| u-\Jj_{0,h} u \|_{L^2(H^1_0(\Omega))} &\lesssim h ( \| u \|_{H^1(H^1_0(\Omega))} + \| u \|_{L^\infty(H^2(\Omega))} )\\
    \| (u-\Jj_{0,h} u)' \|_{L^2(H^{-1}(\Omega))} &\lesssim
    h ( \| u \|_{H^1(H^1_0(\Omega))} + \| u \|_{H^2(H^{-1}(\Omega))} + \| u \|_{L^\infty(H^2(\Omega))} ),\\
    \| (u-\Jj_{0,h} u)' \|_{L^2(L^2(\Omega))} &\lesssim
    h ( \| u \|_{H^1(H^1_0(\Omega))} + \| u \|_{H^2(L^2(\Omega))} + \| u \|_{L^\infty(H^2(\Omega))} ).
  \end{align*}
\end{theorem}
\begin{proof}
  To show the first estimate, in view of Theorem~\ref{thm:tp} it suffices to consider $\| \Jj_{0,h}^\otimes u-\Jj_{0,h} u \|_{L^2(H^1_0(\Omega))}$.
  Note that by Lemma~\ref{lem:int:simp} we have that
  \begin{align*}
    \| \nabla(\Jj_{0,h}^\otimes u-\Jj_{0,h} u) \|_{L^2(K)} &\lesssim h \| \nabla(\Jj_{0,h}^\otimes u)' \|_{L^2(K)}
  \end{align*}
  Summing over all elements $K\in\tT_h$ and applying Lemma~\ref{thm:int:simp:aux} shows the first of the stipulated estimates.
  To show the second and third estimate, we will again apply Theorem~\ref{thm:tp}.
  In order to treat the remaining terms,  note that
  \begin{align*}
    \| (\Jj_{0,h}^\otimes u-\Jj_{0,h} u)' \|_{L^2(H^{-1}(\Omega))} \leq \| (\Jj_{0,h}^\otimes u-\Jj_{0,h} u)' \|_{L^2(L^2(\Omega))}
    \lesssim h \|\nabla(\Jj_{0,h}u)'\|_{L^2(L^2(\Omega))},
  \end{align*}
  where the last estimate follows from Lemma~\ref{lem:int:simp}. Then, we apply Lemma~\ref{thm:int:simp:aux}
  to finish the proof.
\end{proof}
\begin{theorem}\label{thm:int:simp:2}
  There holds
  \begin{align*}
    \| u-\Jj_h u \|_{L^2(H^1(\Omega))} &\lesssim h ( \| u \|_{H^1(H^1(\Omega))} + \| u \|_{L^\infty(H^2(\Omega))} ).
  \end{align*}
\end{theorem}
\begin{proof}
  We note that
  \begin{align*}
    \| (\Jj_h^\otimes -\Jj_h )u\|_{L^2(K)} + \| \nabla(\Jj_h^\otimes -\Jj_h )u\|_{L^2(K)} 
     \lesssim h\|\nabla(\Jj_h^\otimes u)'\|_{L^2(K)}
   \end{align*}
   by Lemma~\ref{lem:int:simp}. 
   The remainder of the proof follows as for Theorem~\ref{thm:int:simp}.
\end{proof}
\subsubsection{Approximating the heat equation in the energy norm}
We have the following result.
\begin{theorem}\label{thm:heat:nrg}
  Let $\Omega$ be a convex polygonal domain.
  Let $u_0\in H^1_0(\Omega)\cap H^2(\Omega)$ and $f\in H^1(L^2(\Omega))$, and $u$ the solution of the heat equation~\eqref{eq:model}.
  Suppose that $\tT_h$ is a simplicial mesh constructed from a tensor product $J_h\otimes \Omega_h$. Then
  \begin{align*}
    \| u - \Jj_{0,h}u \|_{L^2(H^1_0(\Omega))\cap H^1(H^{-1}(\Omega))} = \OO(h).
  \end{align*}
\end{theorem}
\begin{proof}
  It is well known that under the given assumptions, there holds the parabolic regularity
  $u\in L^\infty(H^2(\Omega))$, $u'\in L^\infty(L^2(\Omega))\cap L^2(H^1_0(\Omega))$, and $u''\in L^2(H^{-1}(\Omega))$.
  According to Theorem~\ref{thm:int:simp}, we conclude the statement.
\end{proof}
\subsubsection{Approximating the heat equation in the least squares norm}
\begin{theorem}\label{thm:heat:ls}
  Suppose that $\tT_h$ is a simplicial mesh constructed from a tensor product $J_h\otimes \Omega_h$.
  If $u\in L^2(H^1_0(\Omega))\cap H^1(H^2(\Omega))\cap H^2(L^2(\Omega))\cap L^\infty(H^3(\Omega))$, then
  \begin{align*}
    \| (u-\Jj_{0,h}u, \nabla u - \Jj_h\nabla u) \|_U = \OO(h).
  \end{align*}
\end{theorem}
\begin{proof}
  The definition of the $U$-norm and Theorems~\ref{thm:int:simp} and~\ref{thm:int:simp:2} show
  \begin{align*}
    \| (u-\Jj_{0,h}u, \nabla u - \Jj_h\nabla u) \|_U
    &\leq  \| u - \Jj_{0,h}u \|_{L^2(H^1_0(\Omega))\cap H^1(H^{-1}(\Omega))}
    + \| \nabla u -\Jj_h\nabla u \|_{L^2(L^2(\Omega))}\\
    &\quad+ \| (u-\Jj_{0,h}u)' \|_{L^2(L^2(\Omega))}
    + \| \div (\nabla u - \Jj_h\nabla u) \|_{L^2(L^2(\Omega))}\\
    &\lesssim h \left( 
    \| u \|_{H^1(H^2(\Omega))} + \| u \|_{H^2(L^2(\Omega))} + \| u \|_{L^\infty(H^3(\Omega))}
    \right).
  \end{align*}
\end{proof}

With respect to the regularity requirements of the last theorem, we state the following.
\begin{proposition}\label{prop:heat:reg}
  Let $\Omega\subset\R^d$ with $\partial\Omega$ smooth, or, particularly, $\Omega\subset\R$ an interval.
  Then, if $u_0\in H^1_0(\Omega)\cap H^3(\Omega)$ and $f\in L^2(H^2(\Omega))\cap H^1(L^2(\Omega))$ and
  $f(0) + \Delta u_0 \in H^1_0(\Omega)$, it follows that
  the solution $u$ of the heat equation~\eqref{eq:model} fulfills
  $u\in L^2(H^1_0(\Omega))\cap H^1(H^2(\Omega))\cap H^2(L^2(\Omega))\cap L^\infty(H^3(\Omega))$.
\end{proposition}
\begin{proof}
  It is well known that under the given assumptions, there holds the parabolic regularity
  $u^{(k)}\in L^2(H^{4-2k}(\Omega))$, $k=0,1,2$. It remains to show that $u\in L^\infty(H^3(\Omega))$.
  To that end, consider spatial partial derivatives $D^\alpha$ up to third order $|\alpha|\leq 3$.
  It is clear that $D^\alpha:H^4(\Omega)\rightarrow H^1(\Omega)$ as well as
  $D^\alpha: H^2(\Omega)\rightarrow \widetilde H^{-1}(\Omega)$ are bounded and linear operators.
  Hence, due to Lemma~\ref{lem:schwarz}, $D^\alpha u\in L^2(H^1(\Omega))\cap H^1(\widetilde H^{-1}(\Omega))$,
  and hence also $D^\alpha u\in C(\overline J;L^2(\Omega))$ due to Lemma~\ref{lem:ttrace}.
  We conclude that $u\in C(\overline J; H^3(\Omega))$.
\end{proof}
\section{Numerical results}\label{sec:examples}
In this section we investigate several examples for $d=1$ (Section~\ref{sec:examples2D}) and $d=2$
(Section~\ref{sec:examples3D}).
For all examples we use $J=(0,1)$. We define the estimator
\begin{align*}
  j(u_h,\ssigma_h) =: \eta^2 = \sum_{K\in\tT_h} \eta(K)^2
\end{align*}
where the local error indicators are given by
\begin{align*}
  \eta(K)^2 := \norm{\ssigma_h-\nabla u_h}{L^2(K)}^2 + \norm{\partial_t u_h-\div\ssigma_h-f}{K}^2 
  + \norm{u_h(0)-u_0}{\partial K\cap \{0\}\times \Omega}^2.
\end{align*}
Our adaptive algorithm uses the D\"orfler criterion to mark elements for refinement, i.e.,
find a (minimal) set of elements $\mathcal{M}\subset \tT_h$ such that
\begin{align*}
  \theta \eta^2 \leq \sum_{K\in\mathcal{M}} \eta(K)^2.
\end{align*}
Throughout we use the parameter $\theta = 1/4$ in the case of adaptive refinements.
If an element $K$ is marked for refinement, i.e., $K\in\mathcal{M}$,
it will be (iteratively) subdivided into $2^{d+1}$ son elements using newest vertex bisection (NVB).
In particular, uniform refinement means that each element is divided into $2^{d+1}$ son elements.

In the figures we visualize convergence rates with triangles where the (negative) slope is indicated by a number
besides the triangle. We plot different estimator and respective error quantities over the number of degrees of freedom
$N$. For uniform refinement we have that $h \simeq N^{-1/(d+1)}$.

\subsection{Examples in 1+1 dimensions}\label{sec:examples2D}
Throughout this section we consider problems where $\Omega = (0,1)$.
The initial mesh of the space-time cylinder $J\times \Omega$ consists of four triangles with equal area.

\subsubsection{Example~1}\label{sec:ex1}
For the first example we consider the smooth manufactured solution
\begin{align*}
  u(t,x) = \cos(\pi \,t) \sin(\pi\,x).
\end{align*}
The data $f$ and $u_0$ are computed thereof.
Since the solution is smooth we expect that the overall error converges at a rate $\OO(h)$ which can be observed in
Figure~\ref{fig:ex1}. 
We also see that the overall estimator converges at the same rate. 
Moreover, we observe that the error between $u$ and the approximation $u_h$ at times $t=0,1$ in the $L^2(\Omega)$ norm
and the error between $u$ and $u_h$ in the $L^2(J\times \Omega)$ norm converge at the higher rate $\OO(h^2)$.
\begin{figure}[htb]
  \begin{center}
    \includegraphics[width=0.7\textwidth]{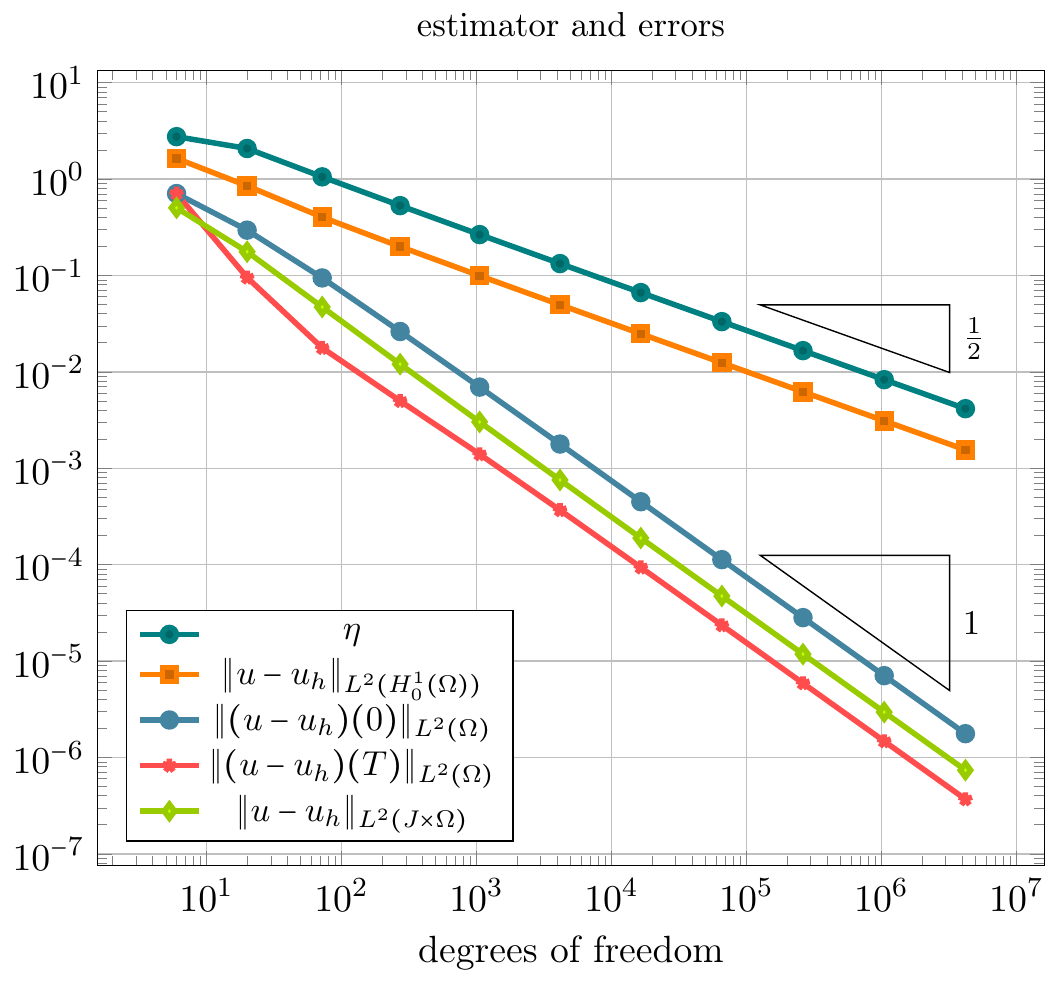}
  \end{center}
  \caption{Estimator and errors for the problem from Section~\ref{sec:ex1}.}
  \label{fig:ex1}
\end{figure}
\subsubsection{Example~2}\label{sec:ex2}

\begin{figure}[htb]
  \begin{center}
    \includegraphics[width=0.7\textwidth]{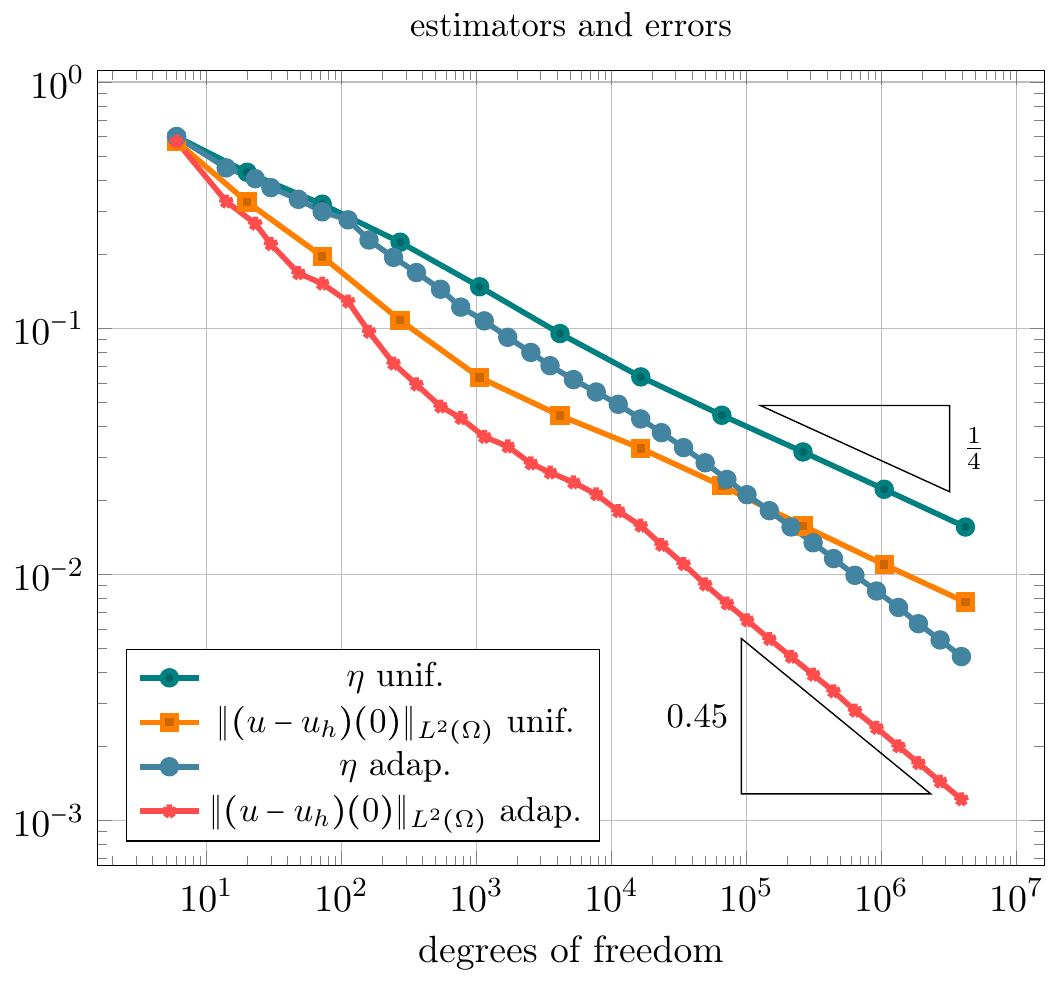}
  \end{center}
  \caption{Estimator and errors for the problem from Section~\ref{sec:ex2}.}
  \label{fig:ex2}
\end{figure}

In this case we choose a constant source $f(t,x) = 1$ and as initial data the ``hat-function''
\begin{align*}
  u_0(x) = 1-2\left\lvert x-\frac12\right\rvert \quad x\in \Omega=(0,1).
\end{align*}
The overall estimator and the error in the initial data is presented in Figure~\ref{fig:ex2}.
It can be observed that uniform refinement leads to a rate of $1/4$ with respect to the overall degrees of freedom both
for the estimator and the error in the initial data whereas in the case of adaptive refinements we obtain a much better
rate of approximately $0.45$.

\subsubsection{Example~3}\label{sec:ex3}

\begin{figure}[htb]
  \begin{center}
    \includegraphics[width=0.7\textwidth]{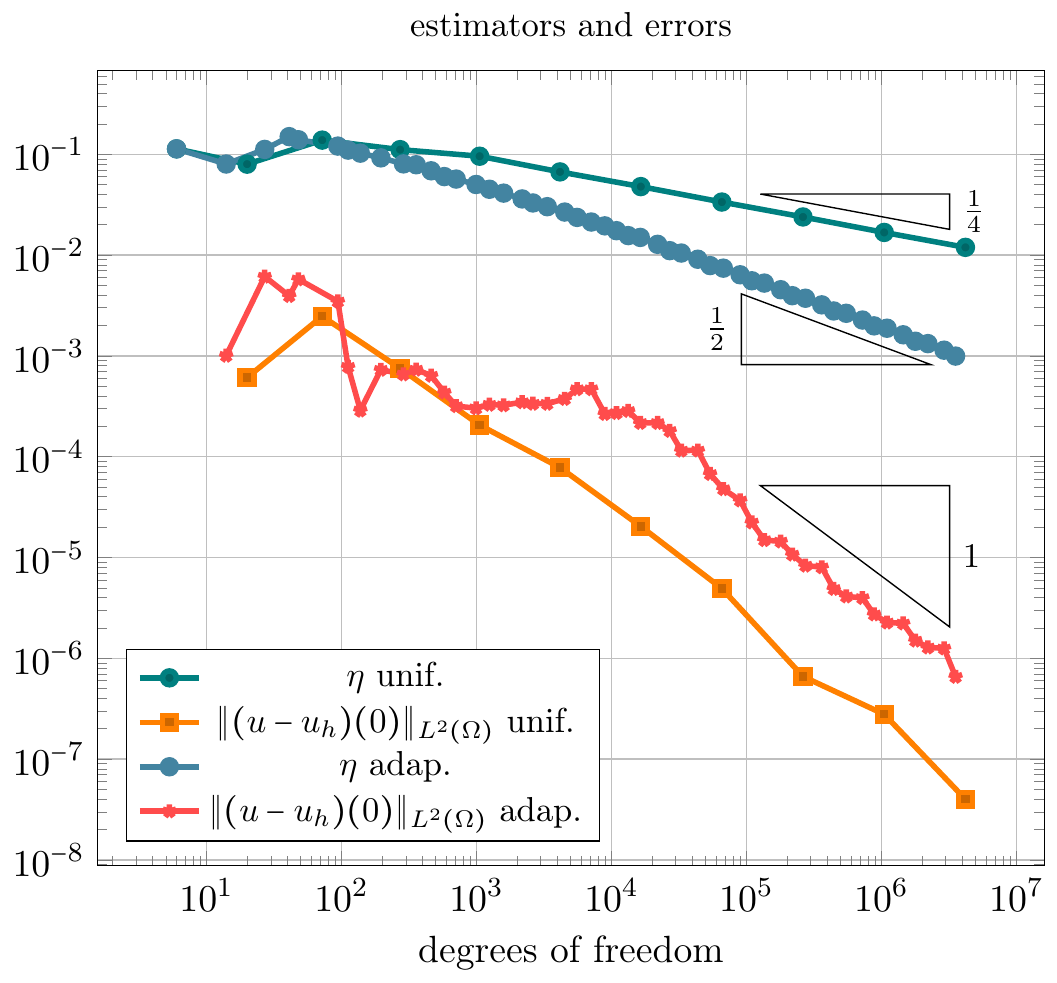}
  \end{center}
  \caption{Estimator and errors for the problem from Section~\ref{sec:ex3}.}
  \label{fig:ex3}
\end{figure}

\begin{figure}[htb]
  \begin{center}
    \includegraphics[width=0.33\textwidth]{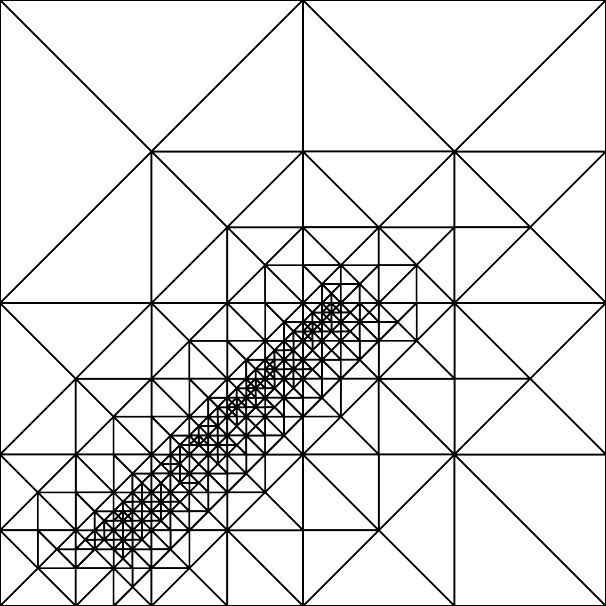}
    \includegraphics[width=0.33\textwidth]{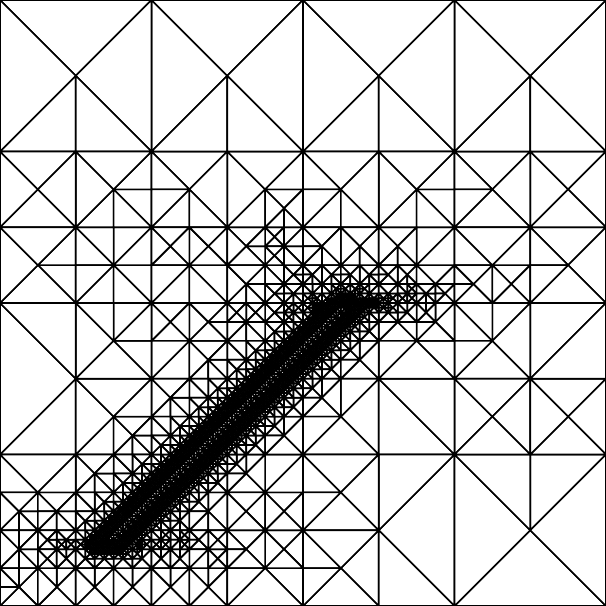}
  \end{center}
  \caption{Adaptively generated meshes with $569$ (left) resp. $13291$ (right) elements for the problem from
  Section~\ref{sec:ex3}. The vertical axis corresponds to the time coordinate.}
  \label{fig:ex3meshSol}
\end{figure}

In this example we consider a problem with homogeneous initial data and
\begin{align*}
  f(t,x) = \begin{cases}
    1 & (t,x) \in \{ (s,y) \in (1/10,1/2)\times (0,1) \,:\, y-1/20\leq s \leq y-1/10 \} \\
    0 & \text{else}
  \end{cases}.
\end{align*}
This function corresponds to a source that is turned on at time $t=0.1$ and turned off at $t=0.5$ and moves with a
constant speed to the right.
The exact solution is not known and in Figure~\ref{fig:ex3} we compare the overall estimator and the error in the
initial time in the cases of uniform and adaptive mesh-refinement.
We observe that in the uniform case we obtain a reduced rate of $1/4$ whereas in the adaptive case a rate of $1/2$ is
recovered for the overall estimator. 
In both cases the error in the initial data converges at the optimal rate.

Figure~\ref{fig:ex3meshSol} shows two examples of meshes generated by the adaptive algorithm.
Stronger refinements around the support of $f$ can be observed.

\subsubsection{Example~4}\label{sec:ex4}

\begin{figure}[htb]
  \begin{center}
    \includegraphics[width=0.7\textwidth]{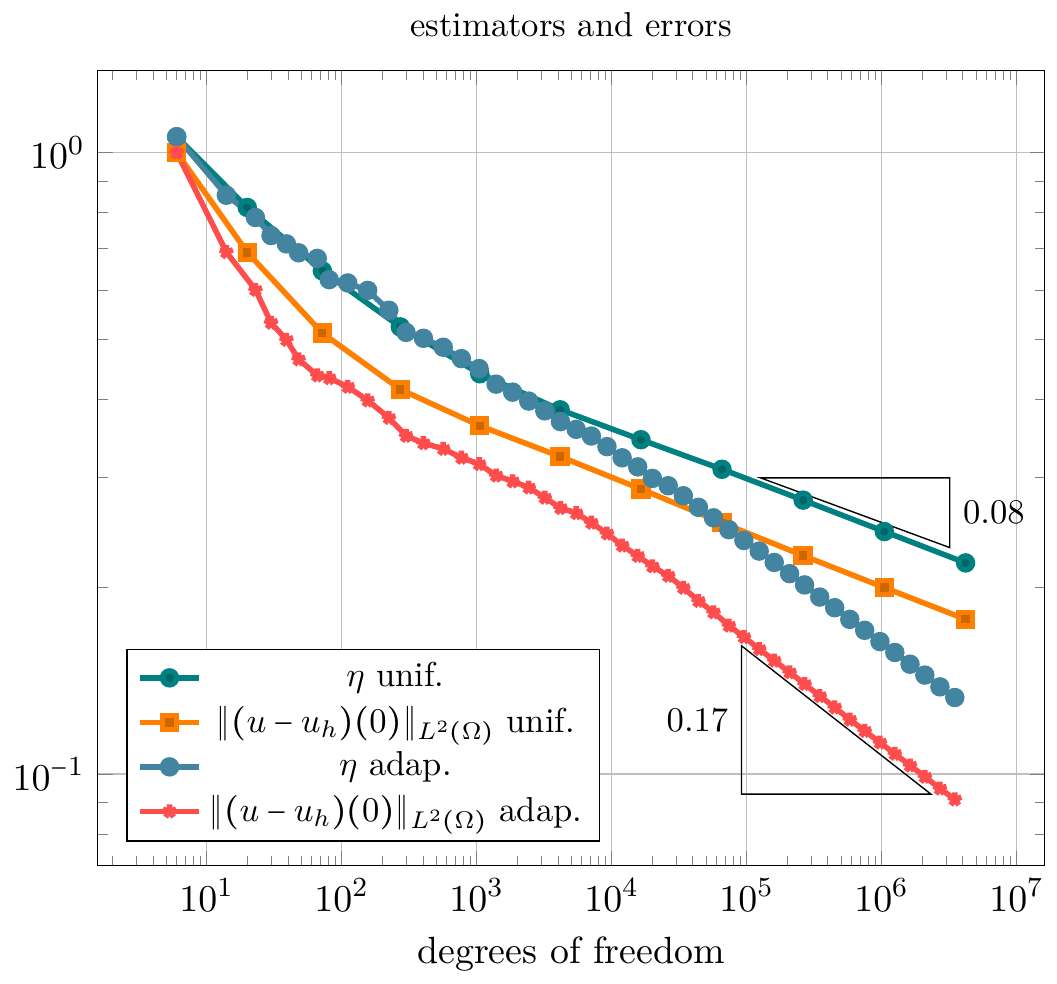}
  \end{center}
  \caption{Estimator and errors for the problem from Section~\ref{sec:ex4}.}
  \label{fig:ex4}
\end{figure}

In this example we set $f(t,x)=2$, $u_0(x) = 1$. Again, the exact solution is not known to us in closed form. 
Note that $u_0$ is regular but does not satisfy homogeneous boundary conditions, i.e., $u_0\in H^1(\Omega)\setminus
H_0^1(\Omega)$.

Figure~\ref{fig:ex4} visualizes the overall estimator and the error at the initial time.
For uniform refinements we observe a rate of $0.08$ and for adaptive refinements the rate is approximately doubled.
We observe a similar behavior for initial data with $u_0\in L^2(\Omega)\setminus H^1(\Omega)$ which is also used
in~\cite[Section~6.3.4]{Andreev_13}. 
We note that reduced rates are also observed in~\cite{Andreev_13}.
\subsection{Examples in 2+1 dimensions}\label{sec:examples3D}

\subsubsection{Example~1}\label{sec:3d:ex1}

\begin{figure}[htb]
  \begin{center}
    \includegraphics[width=0.7\textwidth]{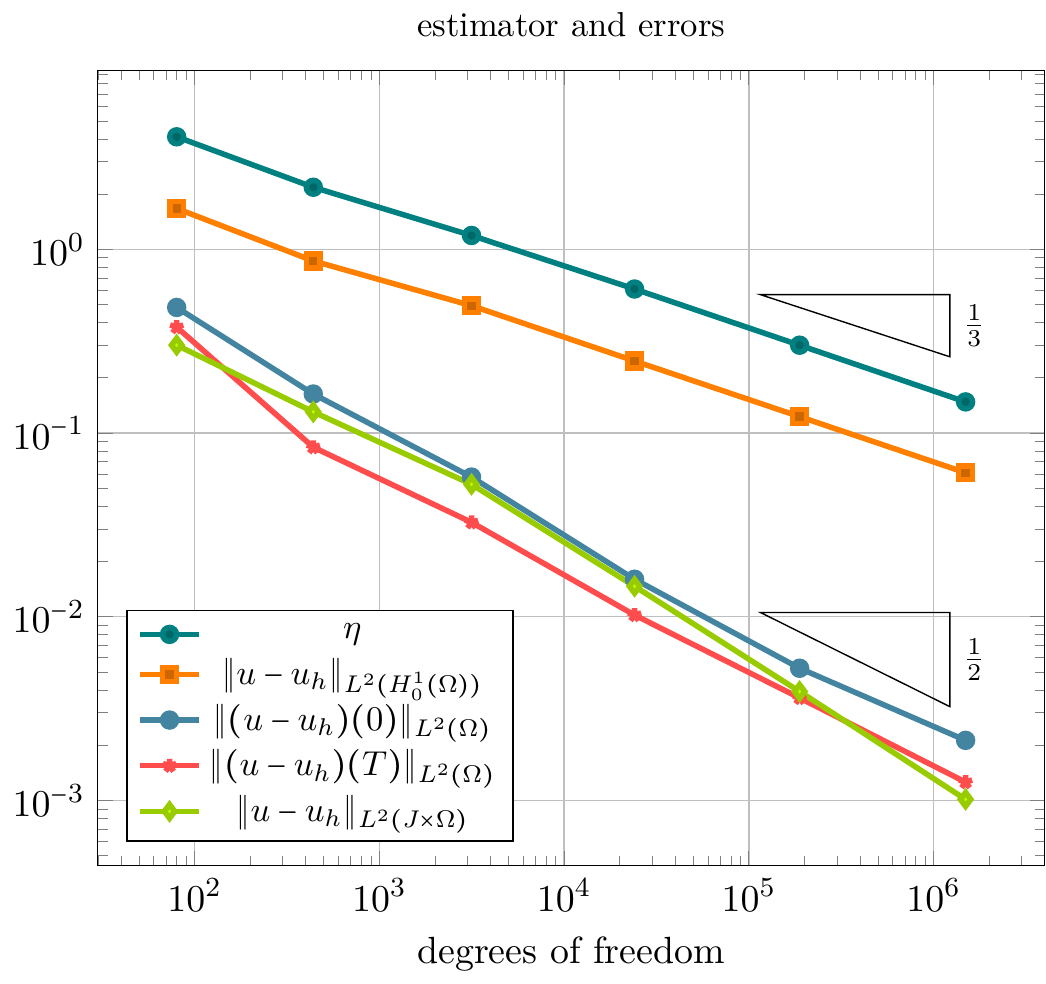}
  \end{center}
  \caption{Estimator and errors for the problem from Section~\ref{sec:3d:ex1}.}
  \label{fig:3D:ex1}
\end{figure}

We consider the domain $\Omega = (0,1)^2$ and the manufactured solution
\begin{align*}
  u(t,x,y) = \cos(\pi t)\sin(\pi x)\sin(\pi y) \quad\text{for } (t,x,y)\in J\times \Omega.
\end{align*}
The data $f$ and $u_0$ are computed thereof.
We note that the solution is smooth and thus we expect for uniform refinements
convergence rates of order $h$ where $h\simeq N^{-1/3}$ and $N$ denotes the overall degrees of freedom.
Figure~\ref{fig:3D:ex1} displays the errors and estimator. One observes the optimal behavior for the overall estimator
and error as well as higher rates of the initial and end error as well as the error in the $L^2$ norm of the space-time
cylinder. We see a behavior of $N^{-1/2}$ which corresponds to $h^{3/2}$.

\subsubsection{Example~2}\label{sec:3d:ex2}

\begin{figure}[htb]
  \begin{center}
    \includegraphics[width=0.7\textwidth]{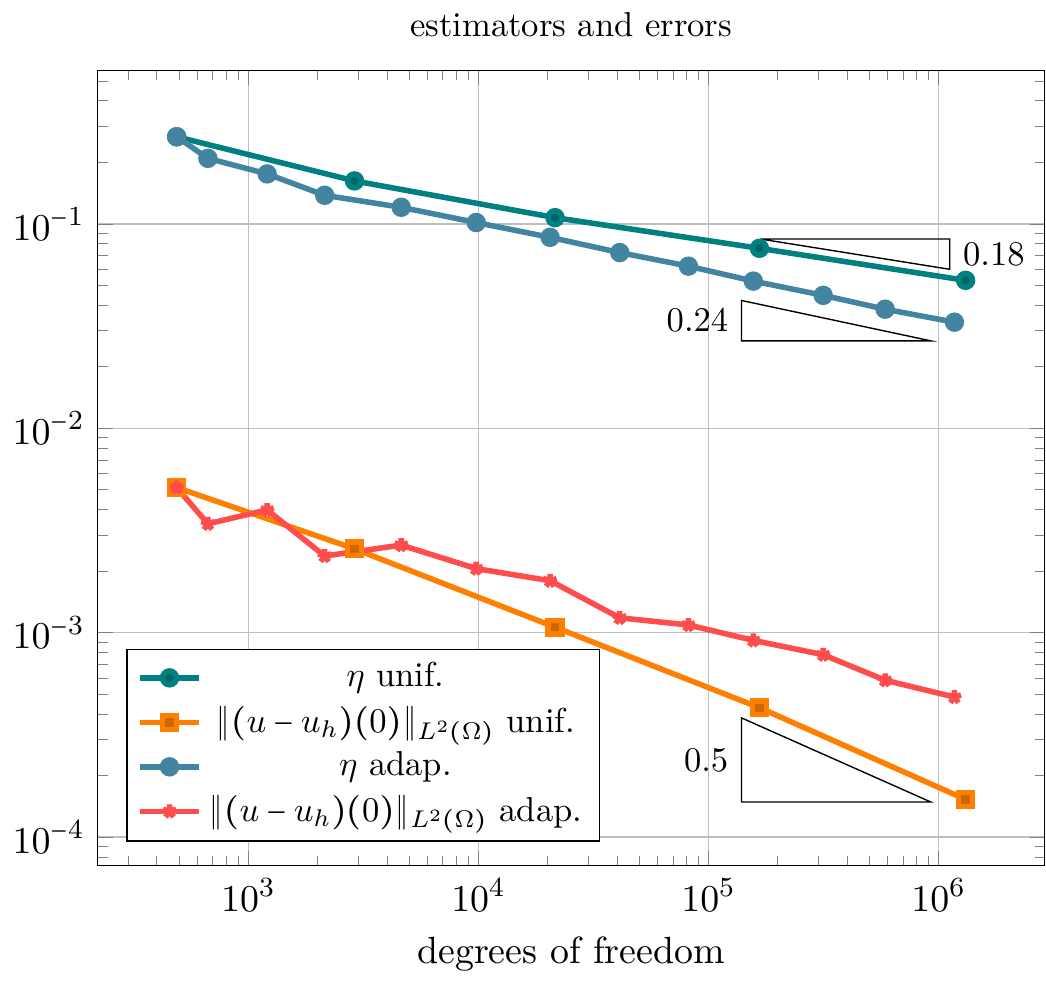}
  \end{center}
  \caption{Estimator and errors for the problem from Section~\ref{sec:3d:ex2}.}
  \label{fig:3D:ex2}
\end{figure}

\begin{figure}[htb]
  \begin{center}
    \includegraphics[width=0.33\textwidth]{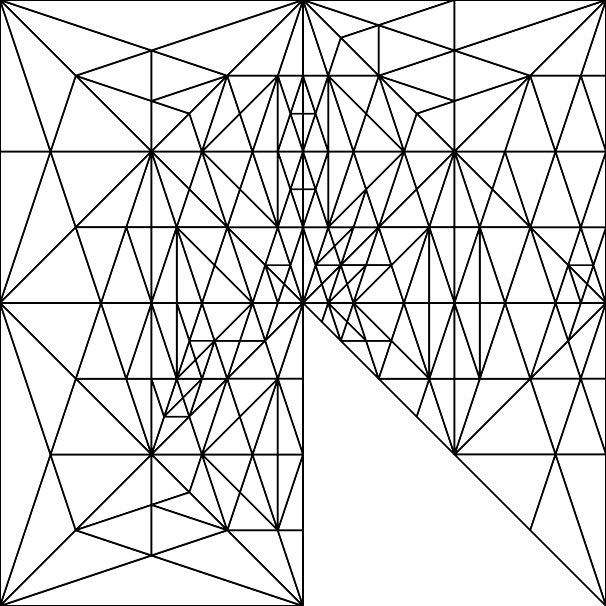}
    \includegraphics[width=0.33\textwidth]{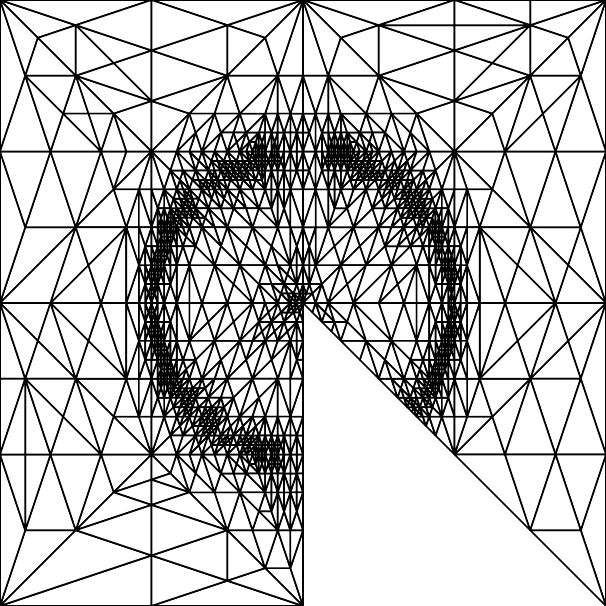}
  \end{center}
  \caption{Space-time meshes restricted to $t=0$ (left) resp. $t=1$ (right) for the problem from
  Section~\ref{sec:3d:ex2}.
  The space-time mesh consists of 149478 elements. The number of boundary elements (triangles) in the left plot is 269
  and in the right plot 1932.}
  \label{fig:3D:ex2mesh}
\end{figure}

For this example we consider $\Omega = (-1,1)^2\setminus \mathrm{conv}\{(0,0),(0,-1),(1,-1)\}$ and $u_0(x,y) = 0$.
The source $f$ is given by
\begin{align*}
  f(t,x,y) = \begin{cases}
    t & \text{if } \sqrt{x^2+y^2}<1/2, \\
    0 & \text{else}.
  \end{cases}
\end{align*}    
Since $\Omega$ has a reentrant corner at the origin we expect that the unknown solution has reduced regularity at this corner.
Figure~\ref{fig:3D:ex2} shows the error in the initial time and the overall estimator for both uniform and adaptive
refinements.
We observe that uniform refinements lead to rates for the overall estimator of approximately $0.2$ whereas for adaptive
refinements we only see a slightly improved rate of approximately $0.24$.

Figure~\ref{fig:3D:ex2mesh} visualizes the boundary of a space-time mesh obtained by the adaptive algorithm at $t=0$
(left plot) and $t=1$ (right plot).
One sees that for $t=0$, where $f$ vanishes, only a few number of elements have been refined.
Contrary, when $t=1$ (right plot), where $f(1,x,y) = 1$ if $x^2+y^2<1/4$, one observes strong refinements close to the
boundary of the support of $f(1,\cdot,\cdot)$ but also towards the reentrant corner.

\subsubsection{Example~3}\label{sec:3d:ex3}

\begin{figure}[htb]
  \begin{center}
    \includegraphics[width=0.7\textwidth]{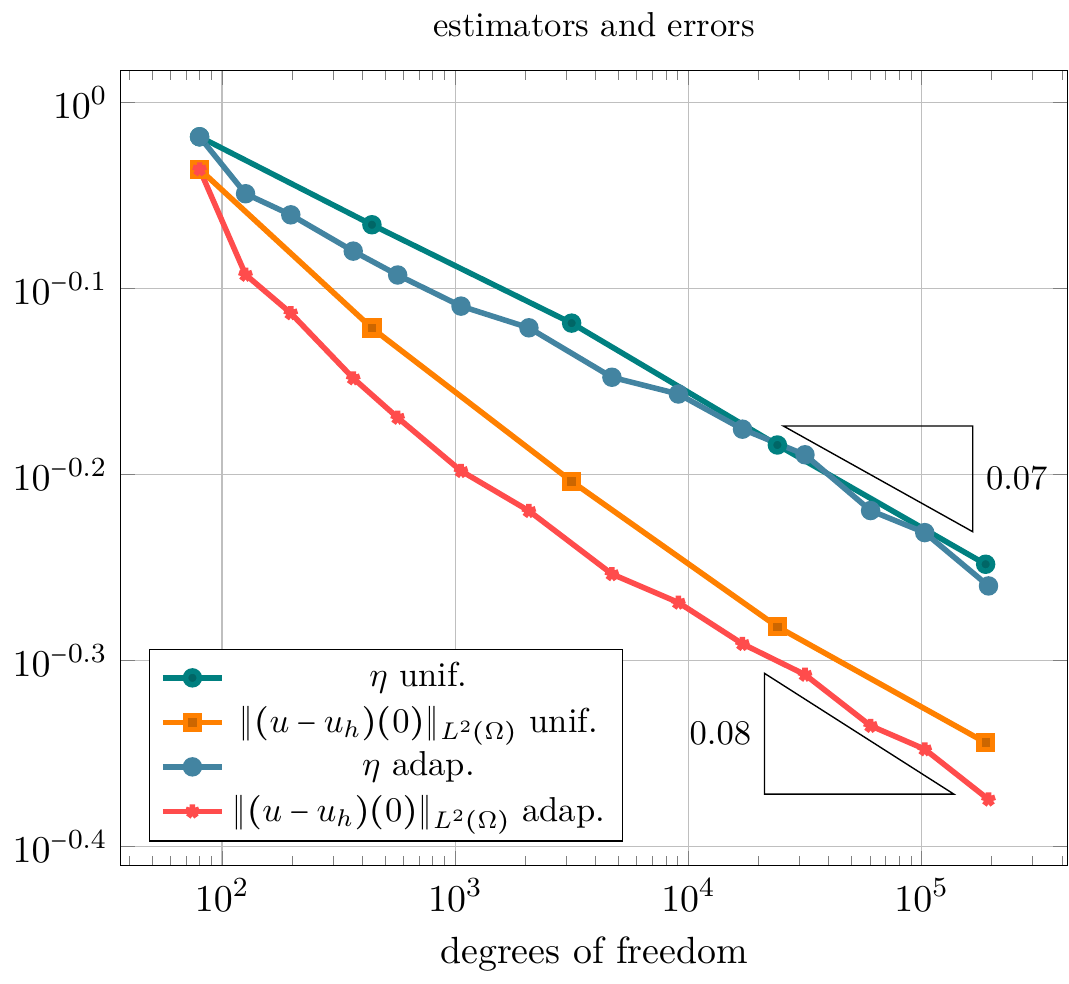}
  \end{center}
  \caption{Estimator and errors for the problem from Section~\ref{sec:3d:ex3}.}
  \label{fig:3D:ex3}
\end{figure}

For this problem we set $f(t,x,y) = 0$, $u_0(x,y) = 1$ and $\Omega = (0,1)^2$.
Figure~\ref{fig:3D:ex3} shows the error in the initial time and the overall estimator for both uniform and adaptive
refinements.
We observe a rate of approximately $0.07$ for uniform refinement, which is not considerably improved using adaptive
refinement.

\bibliographystyle{abbrv}
\bibliography{literature}

\begin{thebibliography}{10}

\bibitem{Andreev_13}
R.~Andreev.
\newblock Stability of sparse space-time finite element discretizations of
  linear parabolic evolution equations.
\newblock {\em IMA J. Numer. Anal.}, 33(1):242--260, 2013.

\bibitem{Andreev_14}
R.~Andreev.
\newblock Space-time discretization of the heat equation.
\newblock {\em Numer. Algorithms}, 67(4):713--731, 2014.

\bibitem{Bey_00}
J.~Bey.
\newblock Simplicial grid refinement: on {F}reudenthal's algorithm and the
  optimal number of congruence classes.
\newblock {\em Numer. Math.}, 85(1):1--29, 2000.

\bibitem{BochevG_09}
P.~B. Bochev and M.~D. Gunzburger.
\newblock {\em Least-squares finite element methods}, volume 166 of {\em
  Applied Mathematical Sciences}.
\newblock Springer, New York, 2009.

\bibitem{BurmanO_18}
E.~Burman and L.~Oksanen.
\newblock Data assimilation for the heat equation using stabilized finite
  element methods.
\newblock {\em Numer. Math.}, 139(3):505--528, 2018.

\bibitem{DautrayL_92}
R.~Dautray and J.-L. Lions.
\newblock {\em Mathematical analysis and numerical methods for science and
  technology. {V}ol. 5}.
\newblock Springer-Verlag, Berlin, 1992.
\newblock Evolution problems. I, With the collaboration of Michel Artola,
  Michel Cessenat and H\'{e}l\`ene Lanchon, Translated from the French by Alan
  Craig.

\bibitem{DevaudS_18}
D.~Devaud and C.~Schwab.
\newblock Space--time hp-approximation of parabolic equations.
\newblock {\em Calcolo}, 55(3):55:35, 2018.

\bibitem{DouglasD_70}
J.~Douglas, Jr. and T.~Dupont.
\newblock Galerkin methods for parabolic equations.
\newblock {\em SIAM J. Numer. Anal.}, 7:575--626, 1970.

\bibitem{CaraMS_17}
E.~Fern\'{a}ndez-Cara, A.~M\"{u}nch, and D.~A. Souza.
\newblock On the numerical controllability of the two-dimensional heat,
  {S}tokes and {N}avier-{S}tokes equations.
\newblock {\em J. Sci. Comput.}, 70(2):819--858, 2017.

\bibitem{FuehrerKarkulik19}
T.~F\"{u}hrer and M.~Karkulik.
\newblock New a priori analysis of first-order system least-squares finite
  element methods for parabolic problems.
\newblock {\em Numer. Methods Partial Differential Equations},
  35(5):1777--1800, 2019.

\bibitem{GunzburgerK_11}
M.~D. Gunzburger and A.~Kunoth.
\newblock Space-time adaptive wavelet methods for optimal control problems
  constrained by parabolic evolution equations.
\newblock {\em SIAM J. Control Optim.}, 49(3):1150--1170, 2011.

\bibitem{HvNVW_16}
T.~Hyt\"{o}nen, J.~van Neerven, M.~Veraar, and L.~Weis.
\newblock {\em Analysis in {B}anach spaces. {V}ol. {I}. {M}artingales and
  {L}ittlewood-{P}aley theory}, volume~63 of {\em Ergebnisse der Mathematik und
  ihrer Grenzgebiete. 3. Folge. A Series of Modern Surveys in Mathematics
  [Results in Mathematics and Related Areas. 3rd Series. A Series of Modern
  Surveys in Mathematics]}.
\newblock Springer, Cham, 2016.

\bibitem{KimPS_18}
D.~Kim, E.-J. Park, and B.~Seo.
\newblock Space-time adaptive methods for the mixed formulation of a linear
  parabolic {P}roblem.
\newblock {\em J. Sci. Comput.}, 74(3):1725--1756, 2018.

\bibitem{LangerMN_16}
U.~Langer, S.~E. Moore, and M.~Neum\"{u}ller.
\newblock Space-time isogeometric analysis of parabolic evolution problems.
\newblock {\em Comput. Methods Appl. Mech. Engrg.}, 306:342--363, 2016.

\bibitem{StarkeM_01}
M.~Majidi and G.~Starke.
\newblock Least-squares {G}alerkin methods for parabolic problems. {I}.
  {S}emidiscretization in time.
\newblock {\em SIAM J. Numer. Anal.}, 39(4):1302--1323, 2001.

\bibitem{StarkeM_02}
M.~Majidi and G.~Starke.
\newblock Least-squares {G}alerkin methods for parabolic problems. {II}. {T}he
  fully discrete case and adaptive algorithms.
\newblock {\em SIAM J. Numer. Anal.}, 39(5):1648--1666, 2001/02.

\bibitem{MNST19}
M.~Montardini, M.~Negri, G.~Sangalli, and M.~Tani.
\newblock Space-time least-squares isogeometric method and efficient solver for
  parabolic problems.
\newblock {\em Math. Comp.}

\bibitem{Moore_18}
S.~E. Moore.
\newblock A stable space-time finite element method for parabolic evolution
  problems.
\newblock {\em Calcolo}, 55(2):Art. 18, 19, 2018.

\bibitem{NK_15}
M.~Neum\"{u}ller and E.~Karabelas.
\newblock Generating admissible space-time meshes for moving domains in
  $d+1$-dimensions.
\newblock Technical report, https://arxiv.org/abs/1505.03973, 2015.

\bibitem{SchwabS_09}
C.~Schwab and R.~Stevenson.
\newblock Space-time adaptive wavelet methods for parabolic evolution problems.
\newblock {\em Math. Comp.}, 78(267):1293--1318, 2009.

\bibitem{Steinbach_15}
O.~Steinbach.
\newblock Space-time finite element methods for parabolic problems.
\newblock {\em Comput. Methods Appl. Math.}, 15(4):551--566, 2015.

\bibitem{SteinbachZ_18}
O.~Steinbach and M.~Zank.
\newblock Coercive space-time finite element methods for initial boundary value
  problems.
\newblock Technical report, Berichte aus dem Institut f\"ur Numerische
  Mathematik, Bericht 2018/7, TU Graz, 2018, 2018.

\bibitem{StevensonW_19}
R.~Stevenson and J.~Westerdiep.
\newblock Stability of galerkin discretizations of a mixed space-time
  variational formulation of parabolic evolution equations.
\newblock Technical report, https://arxiv.org/abs/1902.06279, 2019.

\bibitem{TantardiniV_16}
F.~Tantardini and A.~Veeser.
\newblock The {$L^2$}-projection and quasi-optimality of {G}alerkin methods for
  parabolic equations.
\newblock {\em SIAM J. Numer. Anal.}, 54(1):317--340, 2016.

\bibitem{Thomee}
V.~Thom\'{e}e.
\newblock {\em Galerkin finite element methods for parabolic problems},
  volume~25 of {\em Springer Series in Computational Mathematics}.
\newblock Springer-Verlag, Berlin, second edition, 2006.

\bibitem{Wloka_87}
J.~Wloka.
\newblock {\em Partial differential equations}.
\newblock Cambridge University Press, Cambridge, 1987.
\newblock Translated from the German by C. B. Thomas and M. J. Thomas.

\bibitem{Yosida}
K.~Yosida.
\newblock {\em Functional analysis}.
\newblock Classics in Mathematics. Springer-Verlag, Berlin, 1995.
\newblock Reprint of the sixth (1980) edition.

\bibitem{Zeidler_90}
E.~Zeidler.
\newblock {\em Nonlinear functional analysis and its applications. {II}/{A}}.
\newblock Springer-Verlag, New York, 1990.
\newblock Linear monotone operators, Translated from the German by the author
  and Leo F. Boron.

\end{thebibliography}
\end{document}